\documentclass[10pt,reqno,a4paper]{amsart}

\usepackage[pdftex]{graphicx}
\usepackage{pgfplots}
\usepackage{tikz}  
\usepackage{caption}

\pgfplotsset{compat=1.18}

\usepackage{amsmath}
\usepackage{amssymb,amsthm,hyperref,caption}
\usepackage{xcolor}
\definecolor{meinBlau}{rgb}{0.2,0.2,0.9} 
\definecolor{blau}{rgb}{0,0,0.75} 
\definecolor{rot}{rgb}{0.74,0,0} 
\definecolor{darkgreen}{rgb}{0,0.4,0.1} 
\definecolor{codegreen}{rgb}{0,0.6,0} 
\definecolor{codeblue}{rgb}{0,0.1,0.7} 

\definecolor{myblauA}{rgb}{0,0.1,1} 
\definecolor{myblauB}{rgb}{0,0.2,0.8} 
\definecolor{myblauC}{rgb}{0,0.3,0.7}
\definecolor{myblauD}{rgb}{0,0.4,0.55}
\definecolor{myblauE}{rgb}{0,0.5,0.4}
\definecolor{myblauF}{rgb}{0,0.6,0.3}

\hypersetup{colorlinks,linkcolor=blau,citecolor=blue,urlcolor=meinBlau}
\usepackage{times}
\usepackage{enumitem}
\usepackage{verbatim}
\usepackage{listings}
\usepackage{fancyvrb}

\lstdefinestyle{Rstyle}{
language=R, 
backgroundcolor=\color{white}, 
commentstyle=\color{codegreen}, 
keywordstyle=\color{codeblue}, 
numberstyle=\color{codeblue}, 
stringstyle=\color{codegreen}, 
basicstyle=\ttfamily\tiny, 
morekeywords={TRUE,FALSE}, 
deletekeywords={data,frame,length,as,character}, 
showstringspaces=false 
}

\allowdisplaybreaks

\newtheorem{theorem}{Theorem}
\newtheorem{lem}[theorem]{Lemma}

\newtheorem{prop}{Proposition}

\theoremstyle{definition}

\newtheorem{example}{Example}
\newtheorem{defi}{Definition}

\def\P{{\mathbb {P}}}
\def\E{{\mathbb {E}}}
\def\Var{{\mathbb {V}}}
\DeclareMathOperator{\Cov}{Cov}

\newcommand{\fallfak}[2]{\ensuremath{#1^{\underline{#2}}}}
\newcommand{\auffak}[2]{\ensuremath{#1^{\overline{#2}}}}

\newcommand{\N}{\ensuremath{\mathbb{N}}}

\DeclareMathOperator{\fc}{FC}

\DeclareMathOperator{\nlaw}{\overset{\mathcal{L}}{\neq}}
\DeclareMathOperator{\law}{\overset{\mathcal{L}}{=}}
\DeclareMathOperator{\claw}{\overset{\mathcal{L}}{\rightarrow}}

\DeclareMathOperator{\Po}{Po}
\DeclareMathOperator{\Be}{Be}

\DeclareMathOperator{\Bin}{Bin}

\DeclareMathOperator{\maxP}{maxP}

\newcommand{\fm}{\ensuremath{\mathfrak{m}}}
\newcommand{\fv}{\ensuremath{\mathfrak{v}}}

\newcommand{\veca}[2]{\ensuremath{\binom{#1}{#2}}}

\begin{document}

\author[M.~Kuba]{Markus Kuba}
\address{Markus Kuba\\
Department Applied Mathematics and Physics\\
University of Applied Sciences - Technikum Wien\\
H\"ochst\"adtplatz 5, 1200 Wien} %
\email{kuba@technikum-wien.at}

\title[Card guessing after a single shelf shuffle]{On Card guessing after a single shelf shuffle}


\begin{abstract}
We consider a card guessing game with complete feedback. An ordered deck of
$n$ cards labeled $1$ up to $n$ is shelf-shuffled exactly one time. 
One after the other a single card is drawn from the shuffled deck. The guesser makes has guess and the card is shown until no cards remain. We provide a distributional analysis of the number of correct guesses under the optimal strategy. We re-obtain the previously derived expectation and add a complete description of the distribution. We also obtain a central limit theorem 
for the number $n$ of cards tending to infinity. Furthermore, we discuss an unbalanced, biased shelf shuffle and show how to derive the extend our analysis, also adding the complete position matrix. Finally, a refined analysis of the number of correct guesses is carried out, distinguishing between pure luck guesses and certified correct guesses.
\end{abstract}

\maketitle

\section{Introduction}
\subsection{Shuffling techniques and card guessing games}
The shuffling of cards and associated card guessing games 
have a long and rich history in mathematics at the intersection of combinatorics, probability theory and statistics,  
dating back to work of Markov~\cite{Markov} and Poincare~\cite{Poincare}. Famous shuffling techniques include the riffle shuffle~\cite{BayerDiaconis,DiaconisGraham1981,Gilbert1955}, also called dove-tail shuffling, as well as the shelf shuffle~\cite{Clay2025a,Clay2025b}; we refer the reader to~\cite{Sellke} for references concerning other shuffling techniques such as the Thorpe shuffle and the clumpy shuffle and to~\cite{Kraitchik1942Recreations} for the Monge shuffle. Additionally, we refer to the book of Diaconis and Fulman~\cite{DiaconisFulman2023Shuffling} for many more references. 
The different card shuffling techniques and their analysis are not only of purely mathematical interest. There are applications to the analysis of clinical trials~\cite{BlackwellHodges1957,Efron1971}, fraud detection related to extra-sensory perceptions~\cite{Diaconis1978}, guessing so-called Zener Cards~\cite{OttoliniSteiner2023}, as well as relations to tea tasting and the design of statistical experiments~\cite{Fisher1936,OT2024}. In this work we analyze the shelf shuffle, as introduced by Diaconis, Fulman and Holmes~\cite{DiaconisFulmanHolmes2013}.
Its importance stems from card shuffling machines called shelf-shuffler, which play an important role in the theory of card guessing,
as sch machines have direct practical applications to casino gambling and manufacturing. A standard shuffling machine employs ten shelves. In a shuffling machine the cards are drawn from the top of the deck. Each card is placed, uniformly at random, on a shelf. The decision to place the card
on the top or bottom of the existing pile on the shelf is also done randomly. The shelves
are then removed and the cards are pressed together, forming a shuffled deck. 
Clay~\cite{Clay2025a} analyzed properties of a single shelf shuffle, including a card guessing game with complete feedback.
An ordered deck of $n$ cards labeled $1$ up to $n$ is shelf-shuffled exactly one time.
Then, one after the other a single card is drawn from the shuffled deck. The guesser wants to correctly guess the label of the drawn cards to maximze the number of correct guesses. 
The person gets full feedback after the guess: the card is always shown until no cards remain. Let $X_n$ denote the number of correct guesses in the full feedback card guessing game
after a single shelf shuffle. Clay obtained the expected value of $X_n$ and showed that $\E(X_n)=\frac34\cdot n$. 
We note here in passing that similar questions recently received very much attention and mention a few corresponding works for the riffle shuffle~\cite{Ciucu1998,KSTY2023,NFNW-KT2022,KT2023,Liu2021}
and also refer to the recent work of Tripathi~\cite{Tripathi2026} and of Clay~\cite{Clay2025b} for related problems. The purpose of this work is to extend the analysis of $X_n$ and to obtain the distribution, the variance, as well as a central limit theorem for $X_n$. Additionally, we introduce a biased, asymmetric shelf shuffle
with parameter $p\in(0,1)$ and show how to extend the analysis to this case and also uncovering a phase transition for $p\to 1$. Finally, motivated by other card guessing games~\cite{KP2024} we also discuss a refined analysis of $X_n$: correct guesses are either pure luck $L_n$, as they occur with probability one-half, or they are certified correct guesses $C_n$ with probability one. Thus, the random variable $X_n$ decomposes into the random variables $L_n$ and $C_n$, counting pure luck and certified correct guesses, respectively:
\[
X_n=L_n+C_n.
\]
We study the joint distribution of both random variables, shedding more light into the distribution of $X_n$; in particular, we prove
that exactly $n/2$ guesses are expected to certified correct and $n/4$ to be pure luck, refining $\E(X_n)=\frac34\cdot n$.

\subsection{Notation}
As a remark concerning notation used throughout this work, we always write $X \law Y$ to express equality in distribution of two random variables (r.v.) $X$ and $Y$, and $X_{n} \claw X$ for the weak convergence (i.e., convergence in distribution) of a sequence
of random variables $X_{n}$ to a r.v.\ $X$. Furthermore we use $\fallfak{x}{s}:=x(x-1)\dots(x-(s-1))$ for the falling factorials, and $\auffak{x}{s}:=x(x+1)\dots(x+s-1)$ for the rising factorials, $s\in\N_0$. Moreover, $f_{n} \ll g_{n}$ denotes that a sequence $f_{n}$ is asymptotically smaller than a sequence $g_{n}$, i.e., $f_{n} = o(g_{n})$, $n \to \infty$, and $f_{n} \gg g_{n}$ the opposite relation, 
$g_{n} = o(f_{n})$.

\section{Shelf Shuffling}
We begin by recalling the definition of the shelf shuffle with a single shelf. 

\begin{defi}[Single-shelf shuffle]
Given an ordered deck of $n$ cards labelled one up to $n$. Cards are drawn from the bottom of the deck and placed in a pile or on a shelf. 
When a card is drawn, it is placed at the top or at the bottom of the pile with probability 1/2.
\end{defi}

An important intermediate result for our analysis is the so-called position matrix, encoding the positions of each card after one shelf shuffle.
We collect the result of~\cite{Clay2025a}. 
\begin{lem}[Position matrix of a single shelf shuffle]
\label{lem:posi}
Suppose a deck of $n$ cards is shuffled once in a shelf-shuffler with one shelf.
Let $m_{i,j}$ be the probability that card $i$ lands in position $j$.
Then, we have
\[
m_{i,j}
=
\frac{\binom{i-1}{j-1} + \binom{i-1}{n-j}}{2^i},
\qquad
1 \le i \le n,
\quad
1 \le j \le n,
\]
with the convention that $\binom{0}{0} = 1$ and $\binom{0}{k} = 0$
for all positive integers $k$.
In particular,
$m_{i,j} = 0$ if and only if $1 \le j \le n-i$,
and $m_{i,n-j+1} = m_{i,j}$ for all $i,j$.
\end{lem}

For the reader's convenience and to make this work more self-contained we include the very short argument~\cite{Clay2025a}. 
\begin{proof}
Cards are always drawn from the bottom of the ordered deck. Thus, immediately before selecting the card labelled $i$ we have exactly $n-i$ cards in the pile. The $i$th card is placed at the bottom 
or the top of the pile with probability 1/2. If the card is placed at the top, we select $j-1$ cards from the remaining $i-1$ cards
and place them at the top. If the card is placed at the bottom, we may choose $n-j$ cards from the remaining $i - 1$ cards and place them at the bottom. 
As these events are disjoint, the result readily follows.
\end{proof}
The distribution of the first card $\fc$ is simply given by $m_{i,1}$, $1\le i\le n$:
\begin{equation}
\label{eq:fc}
\P\{\fc=i\}=m_{i,1}=\frac{1+\delta_{i,n}}{2^i},
\end{equation}
where $\delta_{k,\ell}$ denotes the Kronecker-delta. Thus, the optimal strategy follows directly from Lemma~\ref{lem:posi}.

\begin{prop}[Optimal strategy for full feedback]
\label{prop:opti}
Suppose a deck of $n$ cards is given a shelf shuffle with a single shelf,
and the player is given complete feedback. The player should guess card labelled one as the first guess and then always the card immediately following the previously shown card in all other positions.
Once one of card $n-1$ or $n$ is first shown, the player should guess the remaining cards in descending order, and is guaranteed to get these correct.
\end{prop}

The following example highlights the optimal strategy.
\begin{example}
\label{ex:1}
For example, if $n = 20$, say, and the first cards that come out are, in order, 1, 2, 5, 10, 11, 19, the player will have guessed cards 1, 2, 3, 6, 11, 12
under the optimal strategy and have gotten cards 1, 2, 11 correct. After this, the player will guess cards 20, 18, 17, 16, 15, 14, 13, 12, 9, 8, 7, 6, 4, 3 and get them all correct.
\end{example}

\section{Distributional analysis}
In~\cite{Clay2025a} an elaborate probabilistic formulation of the guessing strategy with complete feedback
is given using maps of sequences and in terms of the
filtration generated by the cards that have already appeared. Although natural, 
as the information available to the player is increasing with respect to time, is seems more difficult to perform a distributional analysis using this point of view, apart from the expected value.
We proceed in a different way, setting up a stochastic recurrence relation for the number of correct guesses $X_n$. This recurrence relation is then solved using generating functions. 
The benefit of this point of approach is that it is robust enough to be generalized to an asymmetric shelf shuffle. 

\subsection{Recursive description}
The key to our analysis is the following recurrence relation.
\begin{theorem}
\label{the:distDecomp}
The random variable $X=X_n$ of correctly guessed cards, starting with a deck of $n$ cards, after a one-time shelf shuffle satisfies for $n\ge 2$ the following distributional equation:
\begin{equation}
\label{eqn:Xn_DistEqn}
X_n \law I_1 \big(X^{\ast}_{n-1}+1\big)+(1-I_1)(X^{\ast}_{n-J_n}+J_n-1), 
\end{equation}
with initial values $X_1=1$ and $X_0=0$ and 
\[
\P\{J_n=i\}=\frac{1+\delta_{i-n}}{2^{i-1}},\quad 2\le i\le n.
\]
Here, $X^{\ast}$ denotes an independent copy of $X$ and the random variables appearing in the decomposition are mutually independent, with $I_1\law \Be(1/2)$ denoting a Bernoulli distribution.
\end{theorem}

\begin{proof}
We guess the card labelled one with probability $1/2$ correctly. Otherwise, we draw any card $i$ of the remaining $n-1$ cards with probability~\eqref{eq:fc}.
As pointed out in the optimal strategy in Proposition~\ref{prop:opti} once we have observed card $i$, we know where the cards $1,2,\dots,i-1$are and these guesses are certified correct. Thus, we the only source of randomness are the correct guesses in the remaining $n-i$ cards.
As the shuffle is generated by a sequence of $n-1$ independent Bernoulli trials, this leads to an independent copy of $X^{\ast}$ of $X$ with a smaller number $n-i$ of cards. 
\end{proof}

\subsection{Distributional analysis}
Let $s_n(v)=\E(v^{X_n})$ denote the probability generating function of $X_n$. From Theorem~\ref{the:distDecomp} we obtain 
for $n\ge 2$ by taking expectations the recurrence relation
\begin{equation}
\label{eq:rec1}
s_n(v)=\frac12 v s_{n-1}(v)+\frac12\sum_{k=2}^{n-1}\frac1{2^{k-1}}v^{k-1}s_{n-k}(v) + \frac12\cdot \frac1{2^{n-1}}\cdot 2s_0(v)v^{n-1},
\end{equation}
with initial values $s_0(v)=1$ and $s_1(v)=v$. 
Let $S_(z,v)$ denote the generating function of the $s_n(v)$, defined 
\[
S(z,v)=\sum_{n\ge 1}\E(v^{X_n})z^n.
\]
\begin{lem}
\label{lem:gf}
The bivariate generating function $S(z,v)$ of the number of correct guesses after a single shelf shuffle with complete feedback is given by 
\[
S(z,v)=\frac{zv+\frac{\frac{z^2v}2}{1-\frac{zv}2}}{1-\frac{zv}2-\frac{\frac{z}2}{1-\frac{zv}2}+
\frac{z}2}=\frac{2zv\big(2+(1-v)z\big)}{4-4vz+(v^2-v)z}.
\]
\end{lem}
\begin{proof}
From the recurrence relation~\eqref{eq:rec1} we get first
\begin{equation}
\begin{split}
s_n(v)z^n&=\frac{zv}2\cdot s_{n-1}(v)z^{n-1}+\frac{z}2\sum_{k=1}^{n-1}\frac{v^{k-1}z^{k-1}}{2^{k-1}}\big(s_{n-k}(v)z^{n-k}\big)
-\frac{z}2s_{n-1}(v)z^{n-1}\\
&\qquad  + z\frac{(zv)^{n-1}}{2^{n-1}}.
\end{split}
\end{equation}
Summing up over $n\ge 2$ translates the recurrence relation to the equation
\[
S(z,v)-zv=\frac{zv}2\cdot S(z,v)+\frac{z}2 \frac{S(z,v)}{1-\frac{zv}2}-\frac{z}2 S(z,v)+\frac{z^2v/2}{1-zv/2},
\]
and furthermore to
\[
\Big(1-\frac{zv}2-\frac{\frac{z}2}{1-\frac{zv}2}+\frac{z}2\Big)\cdot S(z,v)=zv+\frac{\frac{z^2v}2}{1-\frac{zv}2}.
\]
This leads to an explicit expression, which is then readily simplified (using a computer algebra system) to the stated form.
\end{proof}

\subsection{Expectation, variance and limit law}
The expected value and the variance of $X_n$ can be obtained from the generating function $S(z,v)$:
\[
\E(X_n)=[z^n]\partial_v S(z,1),\quad \Var(X_n)=[z^n]\big(\partial_v^2 S(z,1)\big) +\E(X_n)-\E(X_n)^2.
\]
Here, $[z^n]$ denotes as usual the extraction of coefficients operator~\cite{FlaSed}.
We also used a short-hand notation
\[
\partial_v S(z,1)= \big(\partial_v S(z,v)\big)|_{v=1}.
\]
After taking the derivative with respect to $v$ and evaluation at $v=1$ we obtain from Lemma~\ref{lem:gf}
\[
\partial_v S(z,1)=\frac{z(z^2-2z+4)}{4(1-z)^2}=z+\sum_{n=2}^{\infty}\frac34\cdot n\cdot z^n.
\]
Similarly, we obtain
\[
\partial_v^2 S(z,1)=\frac{z^2(z^3-2z^2-8)}{8(z-1)^3},
\]
and thus the second factorial moment $\E(\fallfak{X_n}2)$. 
After partial fraction decomposition and the standard identity 
\[
\frac{1}{(1-z)^{m+1}}=\sum_{n\ge 0}\binom{n+m}{m}z^n,
\]
as well as the classical formula
\begin{equation}
\label{eqn:Var}
\Var(X_n)=\E(\fallfak{X_n}2)+\E(X_n)-\E(X_n)^2.
\end{equation}
we arrive at the following result. 
\begin{theorem}
The random variable $X=X_n$ of correctly guessed cards, starting with a deck of $n$ cards, after a one-time shelf shuffle satisfies
\[
\E(X_n)=\frac34\cdot n,\quad n\ge 2, \quad \text{with }\E(X_1)=1.
\]
Furthermore,
\[
\Var(X_n)=\frac{n}{16},\quad n\ge 3,\quad  \text{with }\Var(X_1)=0,\quad \Var(X_2)=\frac14.
\]
\end{theorem}

Next we turn to the limit law for $X_n$. For the derivation of the limit law we use Hwang's quasi-power theorem~\cite{Hwang1994,Hwang1998}.
First, we state our result. 
\begin{theorem}
\label{the:main}
The random variable $X=X_n$ of correctly guessed cards, starting with a deck of $n$ cards, after a one-time shelf shuffle satisfies
a central limit theorem
\[
\frac{X_n-\frac34 n}{\sqrt{n}/4}\to \mathcal{N}(0,1). 
\]
\end{theorem}

For the proof of this result we turn to the quasi-power theorem in the formulation of~\cite{FlaSed}. 
Given a function $f(v)$ analytic at $v = 1$ and assumed to satisfy $f(1) \neq 0$, we set
\begin{equation}
\fm(f) = \frac{f'(1)}{f(1)}, 
\qquad
\fv(f) = \frac{f''(1)}{f(1)} + \frac{f'(1)}{f(1)} - \left( \frac{f'(1)}{f(1)} \right)^2 .
\end{equation}
The quantities $\fm(f)$ and $\fv(f)$ are called the \emph{analytic mean} and \emph{analytic variance} of function $f$~\cite{FlaSed}.
We collect the following result.
\begin{theorem}[Hwang's quasi-powers Theorem]
Let the $X_n$ be non-negative discrete random variables, supported by $\mathbb{Z}_{\ge 0}$, with probability generating functions $p_n(u)$. Assume that,
uniformly in a fixed complex neighbourhood of $v = 1$, for sequences
$\beta_n, \kappa_n \to +\infty$, one has
\begin{equation*}
\E(v^{X_n})=  A(v)\, B(v)^{\beta_n} \left( 1 + O\!\left( \frac{1}{\kappa_n} \right) \right),
\end{equation*}
where $A(u)$, $B(u)$ are analytic at $u = 1$ and $A(1) = B(1) = 1$.
Assume finally that $B(u)$ satisfies the so-called variability condition,
\[
\fv(B(v))=B''(1) + B'(1) - B'(1)^2 \neq 0.
\]
Under these conditions, the mean and variance of $X_n$ satisfy
\begin{equation*}
\mu_n = \mathbb{E}(X_n)
= \beta_n \fm(B(v)) + \fm(A(v)) + O\!\left( \kappa_n^{-1} \right),
\end{equation*}
\begin{equation*}
\sigma_n^2 = \mathbb{V}(X_n)
= \beta_n \fv(B(v)) + \fv(A(v)) + O\!\left( \kappa_n^{-1} \right).
\end{equation*}
The distribution of $X_n$ is, after standardization, asymptotically Gaussian, and the
speed of convergence to the Gaussian limit is $\mathcal{O}\left( \kappa_n^{-1} + \beta_n^{-1/2} \right)$:
\begin{equation*}
\mathbb{P}\left\{
\frac{X_n - \mathbb{E}(X_n)}{\sqrt{\mathbb{V}(X_n)}} \le x
\right\}
=\Phi(x)+\mathcal{O}\left(\frac{1}{\kappa_n}+\frac{1}{\sqrt{\beta_n}}\right),
\end{equation*}
where $\Phi(x)$ is the distribution function of a standard normal,
\[
\Phi(x)=\frac{1}{\sqrt{2\pi}}
\int_{-\infty}^{x}
e^{-v^2/2}\, dv.
\]
\end{theorem}

\begin{proof}[Proof of Theorem~\ref{the:main}]
We analyze the singular structure of $S(z,v)$. The denominator 
\[
4-4vz+(v^2-v)z
\]
has two roots $\rho_1(v)$ and $\rho_2(v)$
\[
\rho_1(v)=\frac{2(v-\sqrt{v})}{v(v-1)},
\quad \rho_2(v)=\frac{2(v+\sqrt{v})}{v(v-1)},
\]
with dominant singularity $\rho(v)=\rho_1(v)$, satisfying 
$\lim_{v\to 1} \rho(v)=1$. The quasi-power theorem applies because the dominant singularity $\rho(v)$ is simple and analytic near $v=1$,.
and $S(z,v)$ is analytic except at $z=\rho(v)$. We expand $S(z,v)$ around $z=\rho(v)$, for $v$ close to $1$, 
obtaining the singular expansion
\[
S(z,v)\sim \frac{A(v)}{1 - \frac{z}{\rho(v)}}, \quad A(v) \neq 0.
\]
By standard contour integration we get the desired expansion of $s_n(v)=\E(v^{X_n})$:
\[
\E(v^{X_n})=[z^n]S(z,v)\sim \frac{A(v)}{\rho(v)^n}.
\]
Application of the Quasi-Power theorem then leads to the stated result; one may also check that the asymptotics $\beta_n\fm(B(v))=3n/4$ and
$\beta_n\fv(B(v))=n/16$.
\end{proof}

Finally, we note that the powerful theorems of Hwang~\cite{Hwang1996LargeDeviationsI} also directly leads to a large deviation principle. 

\section{Extensions}
\subsection{Asymmetric shelf shuffle}
An obvious question is how robust the presented approach is. Consider the following biased asymmetric shelf shuffle
\begin{defi}[Asymmetric Single-shelf shuffle]
Given an ordered deck of $n$ cards labelled one up to $n$ and a parameter $p\in(0,1)$. Cards are drawn from the bottom of the deck and placed in a pile. 
When a card is drawn, it is placed at the top with probability $p$ and at the bottom of the pile with probability $1-p$.
\end{defi}
Our aim is to maximize again the number of correct guesses. We require the distribution of the first drawn card.
\begin{lem}
The distribution of the first card $\fc_n$ in a single asymmetric shelf shuffle with parameter $p\in(0,1)$ is given by as follows:
\begin{equation}
\label{eq:fc2}
\P\{\fc_n=i\}=p(1-p)^{i-1},\quad 1\le i \le n-1
\end{equation}
and 
\[
\P\{\fc_n=n\}=(1-p)^{n-1}.
\]
\end{lem}
\begin{proof}
For $i$, with $1\le i \le n-1$, landing on top, we must have a successful coin toss with probability $p$, and then $i-1$ unsuccessful tosses.
For $n$ ending on top we require all other cards to be below.
\end{proof}
Of course, we can also obtain the complete position matrix, which is double stochastic.
\begin{prop}[Position matrix of a single asymmetric shelf shuffle]
\label{prop:posi}
Suppose a deck of $n$ cards is shuffled once in an asymmetric shelf-shuffler with one shelf with parameter $p$.
Let $m_{i,j}$ be the probability that card $i$ lands in position $j$.
Then, we have
\[
m_{i,j}
=
p^{j}(1-p)^{i-j}\binom{i-1}{j-1} + p^{i-1-(n-j)}(1-p)^{n-j+1}\binom{i-1}{n-j},
\]
for $1 \le i \le n-1$ and $1 \le j \le n$. For $i=n$ we have 
\[
m_{n,j}= p^{j-1}(1-p)^{n-j}\binom{n-1}{j-1},\quad 1\le j\le n.
\]
\end{prop}
Here, we used again the convention that $\binom{0}{0} = 1$ and $\binom{0}{k} = 0$
for all positive integers $k$. In particular, $m_{i,j} = 0$ if and only if $1 \le j \le n-i$,
and $m_{i,n-j+1} = m_{i,j}$ for all $i,j$.

\begin{proof}
We readily adapt the argument from the case $p=1/2$. Immediately before selecting the card labelled $i$ we have exactly $n-i$ cards in the pile. The $i$th card is placed at the bottom with probability $1-p$
or the top of the pile with probability $p$. If the card is placed at the top, we select $j-1$ cards from the remaining $i-1$ cards
and place them at the top, leading to a common factor $p^{j}(1-p)^{i-j}$ and the number of ways to choose the $j-1$ cards. If the card is placed at the bottom, we may choose $n-j$ cards from the remaining $i - 1$ cards and place them at the bottom. This gives a common factor $p^{i-1-(n-j)}(1-p)^{n-j+1}$, which takes into account that $i$ went to the bottom, as well as $n-j$ cards too, and the remaining cards to the top. As these events are disjoint, the result readily follows.
\end{proof}

For $\frac12\le p<1$ we can carry out analysis as before, as the optimal strategy stays the same, but for $p<\frac12$ the structure is asymptotically similar, but the optimal strategy changes, 
as we do not always guess the number one, as the guess depends on the value of $p$. 
Apparently, we have to study the equation $p=(1-p)^{n-1}$, allowing us to distinguish which label appear with the highest probability. Let $\nu=\lfloor \ln(p)/\ln(1-p)\rfloor+1$. Given a value of $n$ with $2\le n\le \nu$, we guess always the largest label $n$, but for all $n>\nu$ we guess the smallest label one. We note again that once card $n-1$ or $n$ is first shown, the player should guess the remaining cards in descending order, and is guaranteed to get these correct. After the first guess, we notice the label $j$ of the shown card. We now immediately know the positions of cards labelled one up to $j-1$. Concerning randomness we replace $n$ by $n-j$ and continue similar to starting with a set of $n-j$ cards by comparing $n-j$ to $\nu$ in order to make our choice for the next guess. Then, we continue the guessing, adding the number $j$ to our guess, as we actually
have the set of labels $j+1,\dots, n$ left. We continue until $n-1$ or $n$ is shown and then correctly guess all the remaining cards.

\begin{example}[Non-uniqueness of optimal strategy]
Consider the case of $n=4$ cards and the equation $p=(1-p)^3$. It has a real solution 
\[
p^{\ast}=
\frac{1}{6}\left(
2^{2/3} \sqrt[3]{\sqrt{93}-9}
- 2^{2/3} \sqrt[3]{9+\sqrt{93}}
+ 6
\right)\approx 0.31767<1/2
\]
and thus the optimal strategy is not unique, as we can either guess one or four. 
\end{example}

\begin{example}
Consider the case of $n=4$ cards. For any $p<p^{\ast}$ we have to guess the label four as the first guess.
\end{example}
We summarize the strategies in the following result.
\begin{theorem}[Optimal strategy for asymmetric shelf shuffle]
For $p\in[\frac12,1]$ we proceed as in the case $p=\frac12$:  we guess the card labelled one as the first guess and then always the card immediately following the previously shown card in all other positions. Once one of card $n-1$ or $n$ is first shown, the player should guess the remaining cards in descending order, and is guaranteed to get these correct.
For $p\in [0,\frac12)$ we guess the label $\maxP_n\in\{1,n\}$, corresponding to the maximum of of $\{p,(1-p)^{n-1}\}$. 
Let $\nu=\lfloor \ln(p)/\ln(1-p)\rfloor+1$. Then, we have 
\[
\maxP_n=
\begin{cases}
n,\quad 2\le n\le \nu,\\
1,\quad n>\nu.
\end{cases}
\]
After the card, say labelled $j$ is shown to the guesser, we proceed in a similar way, but with a reduced set of indices $\{1,\dots,n-j\}$ and guess the label corresponding to the maximum of $\{p,(1-p)^{n-j}\}$, shifted of course by $j$.
\end{theorem}
Thus, the distributional equation has a slightly different form. We can group the two cases together by using a suitably defined indicator variable. 

\begin{theorem}
The random variable $X=X_n$ of correctly guessed cards, starting with a deck of $n$ cards, after an asymmetric one-time shelf shuffle with parameter $p \in(0,1)$ satisfies for $n\ge 2$ the following distributional equation:
\begin{equation}
\label{eqn:Xn_DistEqn3}
X_n \law X^{\ast}_{n-\fc_n}+\fc_n-1 + \mathbb{I}_n(\fc_n=\maxP_n), 
\end{equation}
with initial values $X_1=1$ and $X_0=0$ and the random variable $\fc$ as stated in~\eqref{eq:fc}.
Here, $\mathbb{I}_n(\fc_n=\maxP_n)$ denotes the indicator of the event that the guess is correct and $X^{\ast}$ denotes an independent copy of $X$ and the random variables appearing in the decomposition 
are mutually independent.
\end{theorem} 
As mentioned before, for $p\ge \frac12$ we always have $\maxP_n=1$, and the distributional equation can be written as
\begin{equation}
\label{eqn:Xn_DistEqn2}
X_n \law I_1 \big(X^{\ast}_{n-1}+1\big)+(1-I_1)(X^{\ast}_{n-J_n}+J_n-1), 
\end{equation}
with $I_1\law \Be(p)$ denoting a Bernoulli distribution and initial values $X_1=1$ and $X_0=0$ and 
\[
\P\{J_n=i\}=p(1-p)^{i-2},\quad 2\le i\le n-1,\quad \P\{J_n=n\}=(1-p)^{n-2}.
\]
The generating functions approach works directly \emph{exactly the same} and we obtain for 
\[
S(z,v)=S(z,v;p)=\sum_{n\ge 1}\E(v^{X_n})z^n
\]
the following result.
\[
S(z,v)=\frac{-z\big(1 + (v - 1)(p - 1)z\big)v}{zv-1 + pv(v - 1)(p - 1)z^2 }
\]
We obtain by extraction of coefficients from $\partial_v S(z,1)$ and $\partial_v^2 S(z,1)$ for the expected value and the variance for $n\ge 3$ and $\frac12\le p<1$ the results
\begin{align*}
\E(X_n)&=(1-p+p^2)n+3p-1-2p^2,\\
\Var(X_n)&=(1-p)p(3p^2-3p+1)n+p(1-p)(10p-8p^2-3).
\end{align*}
The denominator is readily factorized and we obtain the dominant singularity $\rho(v)$, depending on $\frac12\le p<1$
\[
\rho(v)=
\frac{-v + \sqrt{(1+4p(p-1))(v^2 - v)+v}}{2vp(pv - p - v + 1)}.
\]
We can readily obtain a central limit theorem using the quasi-power theorem, such that
\[
\frac{X_n-\mu_n}{\sigma_n}\to\mathcal{N}(0,1),
\]
with $\mu_n=\E(X_n)$, $\sigma_n^2=\Var(X_n)$; details are omitted. Of interest is also the transition at $p=1$. 
Here, we deterministically have $X_n=n$. Thus, we set $p=1-\frac{\lambda}{n^{\alpha}}$, with real $\alpha,\lambda>0$, which approaches one as $n$ tends to infinity. The transition can be easily observed by looking at the probability $\P\{X_n=n\}$, as it satisfies for $p=1-\frac{\lambda}{n^{\alpha}}$, $\alpha,\lambda>0$ and $n\to \infty$ the expansions
\[
\P\{X_n=n\}\to
\begin{cases}
1,\quad 0<\alpha<1,\\
e^{-\lambda},\quad \alpha=1,\\
0,\quad \alpha\ge 1,
\end{cases}
\]
which follows directly from the explicit formula $\P\{X_n=n\}=p^n$ and the $\exp-\log$ expansion.
More refined, we additionally anticipate here a phase transition for the shifted random variable $Z_n=n-X_n$ depending on the value of $\alpha$ from degenerated, to a discrete distribution, 
and then to a normal distribution. We expect a classical $\Po(\lambda)$ discrete limit in the case $\alpha=1$,
motivated by the corresponding expansions of the expected value and the variance in this range
\[
\E(X_n)\sim n-c+\mathcal{O}\big(\frac1n\big),\quad \Var(X_n)\sim c+\mathcal{O}\big(\frac1n\big)
\]
This can be made precise by a closer inspection of the probability generating function using the dominant singularity; details are again omitted.

\smallskip 

In the case of $0<p<\frac12$ the analysis is more involved, depending on $\nu=\lfloor \ln(p)/\ln(1-p)\rfloor+1$, where the optimal strategy changes.
First, we note that the distributional equation can be written as
\begin{equation}
\label{eqn:Xn_DistNu1}
X_n \law I_1 X^{\ast}_{n-1}+(1-I_1)\big(X^{\ast}_{n-J_n}+J_n-1+\mathbb{I}_n(J_n=n)\big), \quad n\le \nu,
\end{equation}
and for $n\ge \nu$ by 
\begin{equation*}
X_n \law I_1 \big(X^{\ast}_{n-1}+1\big)+(1-I_1)(X^{\ast}_{n-J_n}+J_n-1).
\end{equation*}
For the interested reader we sketch how to determine the distribution of $X_n$. In order to solve this, we introduce an auxiliary random variable $Y_n$, which satisfies~\eqref{eqn:Xn_DistNu1} for all $n\ge 2$.
Consequently, 
\[
X_n\law Y_n\,\text{ for } \,1\le n\le \nu;\qquad X_n\nlaw Y_n\, \text{ for } n>\nu. 
\]
Let $T=T(z,v)=\sum_{n\ge 1}\E(v^{Y_n})z^n$. We obtain the equation 
\[
T-zv=zpT + \frac{zpT}{1-zv(1-p)}-zpT+\frac{z^2v^2(1-p)}{1vz(1-p)},
\]
which leads to a very simple explicit expression for $T$:
\[
T(z,v)=\frac{z v} {1-z\big(p+(1-p)v\big)},
\]
such that $Y_n$, and consequently also $X_n$ for $n\le \nu$, is given by a sum of $n-1$ iid Bernoulli random variables $B_k\law \Be(1-p)$ shifted by one:
\[
X_n\law 1+\sum_{k=1}^{n-1}B_k\law 1+\Bin(n-1,1-p),\quad 1\le n\le \nu.
\]
Then, we use again generating functions for treating~\eqref{eqn:Xn_DistEqn3}.
As before, let $S=S(z,v)=\sum_{n\ge 1}\E(v^{X_n})z^n$, which splits into two parts
\begin{align}
\label{eqn:decompo}
S(z,v)=S_{\le \nu}(z,v)+S_{>\nu}(z,v)=\sum_{n=1}^{\nu}\E(v^{X_n})z^n+\sum_{n> \nu}\E(v^{X_n})z^n,
\end{align}
with $S_{\le \nu}(z,v)=\sum_{n=1}^{\nu}\E(v^{X_n})z^n$ already known from $T(z,v)$:
\begin{equation}
\label{eqn:Lenny}
S_{\le \nu}(z,v)=\sum_{n=1}^{\nu}\E(v^{Y_n})z^n=zv\frac{1-z^{\nu}\big(p+(1-p)v\big)^{\nu}}{1-z\big(p+(1-p)v\big)}.
\end{equation}
From~\eqref{eqn:Xn_DistEqn3} we obtain
the equation
\begin{align*}
S(z,v)-zv&=pz\big(S_{\le \nu}(z,v)-(1-v)s_{\nu}(v)+v S_{>\nu}(z,v)\big)\\
&\quad+ \frac{zpS(z,v)}{1-zv(1-p)}-z p S(z,v)+\sum_{n\ge 2}(1-p)^{n-1}v^{n-1+\mathbb{I}_n(n\le \nu)}z^n.
\end{align*}
This allows to characterize $S_{>\nu}(z,v)$, and thus the distribution of $X_n$
in terms of the $S_{\le \nu}(z,v)$~\eqref{eqn:Lenny} and $s_{\nu}(v)=v(p+v(1-p))^{\nu-1}$:
\begin{align*}
S_{>\nu}(z,v)&=\frac{zv +(pz-1) S_{\le \nu}(z,v)-p(1-v)zs_{\nu}(v) + \frac{zpS_{\le\nu}(z,v)}{1-zv(1-p)} -zpS_{\le\nu}(z,v)
}{1-pzv - \frac{zp}{1-zv(1-p)}+z p}\\
&\quad +\frac{1}{1-pzv - \frac{zp}{1-zv(1-p)}+z p}\sum_{n\ge 2}(1-p)^{n-1}v^{n-1+\mathbb{I}_n(n\le \nu)}z^n.
\end{align*}

\subsection{Refined analysis of the ordinary shelf shuffle}
As outlined in the introduction, we turn to a refined analysis of $X_n$, by keeping track of the pure luck $L_n$ and the certified correct guesses $C_n$. We continue and refine Example~\ref{ex:1}.

\begin{example}
\label{ex:2}
If $n = 20$ the first cards that come out are, in order, 1, 2, 5, 10, 11, 19, the player will have guessed cards 1, 2, 3, 6, 11, 12
under the optimal strategy and have gotten cards 1, 2, 11 correct. After this, the player will guess cards 20, 18, 17, 16, 15, 14, 13, 12, 9, 8, 7, 6, 4, 3 and get them all correct. Thus, here we have three pure luck guesses and 14 certified correct guesses.
\end{example}

Of course, the distributional equation of Theorem~\ref{the:distDecomp} is readily refined to keep track of the random vector $(L_n,C_n)$:
\begin{equation}
\label{eqn:DistNew1}
\veca{L_n}{C_n} \law I_1 \veca{L^{\ast}_{n-1}+1}{C^{\ast}_{n-1}} +(1-I_1)\veca{L^{\ast}_{n-J_n}}{C^{\ast}_{n-J_n}+J_n-1}, 
\end{equation}
for $n\ge 2$ with $L_1=L_0=0$, $C_1=1$ and $C_0=0$. Here, we have again
\[
\P\{J_n=i\}=\frac{1+\delta_{i-n}}{2^{i-1}},\quad 2\le i\le n.
\]
Let $s_n(v,w)=\E(v^{L_n}w^{C_n})$, such that 
\[
s_n(v,v)=\E(v^{L_n+C_n})=\E(v^{X_n}).
\]
We obtain for $n\ge 2$ by taking expectations the recurrence relation
\begin{equation}
\label{eq:recRefined1}
s_n(v,w)=\frac12 v s_{n-1}(v,w)+\frac12\sum_{k=2}^{n-1}\frac1{2^{k-1}}w^{k-1}s_{n-k}(v,w) + \frac12\cdot \frac1{2^{n-1}}\cdot 2s_0(v,w)w^{n-1},
\end{equation}
with initial values $s_0(v,w)=1$ and $s_1(v,w)=w$. Let $S_(z,v,w)$ denote the generating function of the $s_n(v,w)$, defined by
\[
S(z,v,w)=\sum_{n\ge 1}\E(v^{L_n}w^{C_n})z^n.
\]
Summing up over $n\ge 2$ translates the recurrence relation to the equation
\[
S(z,v,w)-zw=\frac{zv}2\cdot S(z,v)+\frac{z}2 \frac{S(z,v,w)}{1-\frac{zw}2}-\frac12 S(z,v,w)+\frac{z^2w/2}{1-zw/2},
\]
leading to the following result.

\begin{lem}
\label{lem:gfMultiVar}
The trivariate generating function $S(z,v,w)$ of the number of correct guesses decomposed into pure luck guesses and certified correct guesses after a single shelf shuffle with complete feedback is given by 
\[
S(z,v,w)=\frac{zw+\frac{\frac{z^2w}2}{1-\frac{zw}2}}{1-\frac{zv}2-\frac{\frac{z}2}{1-\frac{zw}2}
+\frac{z}2}=\frac{2zw\big(2 + (1 - w)z\big)}{4 + w(v - 1)z^2 -2z(v +w)}.
\]
\end{lem}
From this generating function we readily obtain various results. 
First, we turn to the expectation, variance and covariance of $L_n$ and $C_n$. 
They are obtained by extraction of coefficients from 
\[
\partial_vS(z,1,w),\quad \partial_wS(z,v,1),
\]
as well as
\[
\partial_v^2S(z,1,w),\quad \partial_w^2S(z,v,1),\quad \partial_w\partial_v S(z,1,1),
\]
using the identity~\eqref{eqn:Var}. We omit the computations.
\begin{theorem}
\label{the:refined}
The random variables $L_n$ and $C_n$, counting the pure luck $L_n$ and the certified correct guesses $C_n$ 
in the card guessing game with full feedback after a single shelf shuffle 
satisfy
\[
\E(L_n)=\frac{n}4,\quad \E(C_n)=\frac{n}2,\quad n\ge 2.
\]
Moreover,
\[
\Var(L_n)=\frac{5n}{16},\quad \Var(C_n)=\frac{n}4,\quad \Cov(L_n,C_n)=-\frac{n}4,\quad n\ge 3.
\]
\end{theorem}
Finally, we note that the denominator of $S(z,v,w)$ has a particularly simple form:
\begin{equation*}
\begin{split}
 w(v - 1)z^2 -2z(v +w)+4&= w(v - 1)\big(z-\rho_1(v,w)\big)\big(z-\rho_2(v,w)\big)\\
&=4\big(1-z/\rho_1(v,w)\big)\big(1-z/\rho_2(v,w)\big),
\end{split}
\end{equation*}
with
\[
\rho_1(v,w)=\frac{v+w-\sqrt{(v-w)^2+4w}}{(v-1)w},\quad \rho_2(v,w)=\frac{v+w+\sqrt{(v-w)^2+4w}}{(v-1)w}.
\]
The dominant singularity around $v=w=1$ is $\rho=\rho_1(v,w)$, which satisfies 
\[
\lim_{v,w\to 1}\rho(v,w)=1.
\]
Here, one may apply the higher dimensional quasi-power theorem analogs of Heuberger and Kropf~\cite{HeubergerKropf2018}, 
leading to a central limit theorem, with variance-covariance governed by~\eqref{the:refined}. We omit the details.

\subsection{Refined analysis of the asymmetric shelf shuffle}
Finally, we note that both ideas, asymmetric shelf shuffle and refined number of correct guesses, can be combined. 
For $\frac12\le p<1$ this leads to the distributional equation
\begin{equation}
\label{eqn:DistNew2}
\veca{L_n}{C_n} \law I_1 \veca{L^{\ast}_{n-1}+1}{C^{\ast}_{n-1}} +(1-I_1)\veca{L^{\ast}_{n-J_n}}{C^{\ast}_{n-J_n}+J_n-1}, 
\end{equation}
for $n\ge 2$ with $L_1=L_0=0$, $C_1=1$ and $C_0=0$, and $I_1=\Be(p)$. We only state the corresponding generating function
$S(z,v,w)=S(z,v,w,p)$:
\[
S(z,v,w)=\frac{zv\big(1 + (v - 1)(-1 + p)z\big)}{1 - pv(w - 1)(-1 + p)z^2 - (1 - p)v + wpz},
\]
from which results for the expected values, variances and the covariance can readily be obtained, as well as distributional results. It is also possible to extend this analysis to the case $0<p<\frac12$.

\section{Outlook and summary}
We considered a card guessing game with complete feedback after a single shelf shuffle. We revisited the optimal strategy for maximizing the number of correct guesses and carried out a complete distributional analysis. Concerning generalizations and extensions for future research, 
we outlined how to extend the analysis to an asymmetric shelf shuffle, as well as to a refined number of correct guesses. 

\smallskip 

A corresponding analysis of the guessing game without feedback
seems to be more involved, insofar as even the optimal strategy is unclear. A spectral analysis of the position matrix for the asymmetric model is doable. Of interest is also the transition for $p\to 0$ in the asymmetric model and to look again for transitions from normal to say Poisson or degenerated.

\section*{Declarations of interest and Acknowledgments}
There are no competing financial or personal interests that influenced the work reported in this paper. 

\bibliographystyle{cyrbiburl}
\bibliography{CardGuessingShelf-refs}{}


\end{document}